\documentclass[12pt]{article}
\usepackage[cp1251]{inputenc}
\usepackage{graphicx}
\usepackage[english,russian]{babel}
\usepackage{amsmath,amssymb,euscript, amsthm,amsfonts,mathrsfs,amscd}
\theoremstyle{plain}
\newtheorem{thm}{Theorem}
\newtheorem{lem}{Lemma}
\newtheorem{prop}{Proposition}
\usepackage{color}
\newtheorem{cor}{Corollary}
\theoremstyle{definition}
\newtheorem{deff}{Definition}

\theoremstyle{remark}
\newtheorem{rem}{Remark}

\usepackage{mathrsfs}
\usepackage[ left=20mm, top=20mm, right=20mm, bottom=20mm, footskip=10mm, nohead, nomarginpar, driver=xetex]{geometry}

\title{
\bf On some limit theorems following from Smith's Theorem}
\author{G.~A.~Zverkina}

\begin{document}
\selectlanguage{english}
\maketitle
\begin{abstract}
We prove an ergodic theorem for a linearwise Markov process, and we give estimation for the average of renewal process amount of overjump; these results are based on Smith's Key Renewal Theorem.

Also we discuss the possibility of using these results for studying some type of queueing systems.
\end{abstract}

\section{Introduction}
Smith's Key Renewal Theorem is the basis for the development of Renewal Theory.
This theorem implies many useful consequences, see, e.g., \cite{smith}, \cite{GBS}, \cite{GK}, etc.
Here we present some consequences of Smith's theorems useful for queueing theory.
We prove the ergodic theorem for linearwise Markov process, as well as we find an estimation of the expectation of renewal process amount of underjump (or backward renewal time), and we give an estimation for the expectation of renewal process amount of overjump (or forward renewal time).

This paper is organized as follows.

In Section 2 we recall the basic facts of renewal theory, and formulate Smith's Key Renewal Theorem in its most general form.

In Section 3 we recall the notion of linearwise (or lineartype) Markov process, formulate and prove the ergodic theorem for linearwise Markov process.

In Section 4 we give estimation for renewal process amount of overjump expectation.

In Section 5 we discuss the possibility of using the results of Sections 3-4 to prove the ergodic theorem concerning multidimensional piecewise-linear Markov processes.

\section{Smith's Key Renewal Theorem}
\subsection{Definitions}
\begin{deff} Renewal process is a \emph{random} increasing sequence $0<t_1<t_2< \ldots<t_n< \ldots$, where $t_{i+1}-t_i= \xi_i$ ($i \in \mathbb{N}$) are non-negative random variables with distribution function $F(s)$, and $t_1$ is non-negative random variable $ \xi_0$ with distribution function $F_1(s)$; random variables $ \xi_i$ are mutually independent.
The times $t_i$ are called {\it renewal times} (points) or {\it jump times} and the intervals $(t_i,t_{i+1})$ are called {\it renewal intervals}.
\end{deff}
\begin{rem}
Sometimes renewal process is defined differently, but the meaning of these definitions remains the same -- see, e.g., \cite{Cox}, \cite{GBS}, \cite{GK}.
\end{rem}
\begin{deff}
Let $N(t) \stackrel{ \text{ \rm def}}{=\!\!\!=} \max\{i:\;t_i<N\}$ be a number of renewals till a time $t$ (or number of jumps observed up to some time $t$). Let us denote $ H(t)\stackrel{\text{\rm def}}{=\!\!\!=} \mathbf{E}\,N(t)=F_1(t)+F_1 \ast \left(\sum \limits _{i=1}^ \infty F^{i \ast} \right)(t)$; $ \mathbf{E}\,N(t)=H(t)$ is called {\it renewal function} (here ``$\ast$'' is a symbol of convolution).
\end{deff}
\begin{deff}
A random variable is called {\it lattice} if all its possible values are in the set $\{a+nb,\;n \in \mathbb{Z}\}=a+b \mathbb{Z}$, $b\neq 0$.
\end{deff}
\begin{deff} A non-negative function $g(x):[0,+\infty)\to [0,+\infty)$ is called directly Riemann integrable if
$$
\lim\limits_{\Delta\downarrow 0}\Delta \sum \limits _{
n=1}^{ \infty} \sup \limits _{x \in[ \Delta(n-1), \Delta n]}g(x)=\lim\limits_{\Delta\downarrow 0} \Delta \sum \limits _{
n=1}^{ \infty} \inf \limits _{x \in[ \Delta(n-1), \Delta n]}g(x)\in(+\infty,+\infty).
$$
\end{deff}
\begin{deff}
Let the sequence $0=t_0<t_1<t_2<\cdots <t_i<t_{i+1}<\cdots$ be a renewal process.

For $t>0$ denote by $x^\ast_t$ the time elapsing from the instant $t$ to the next renewal, and by $x_t$ the time elapsed from the previous renewal to $t$. Thus, if $n$ is defined by the condition $t_n\leqslant t<t_{n+1}$, then $x^\ast_t = t_{n+1} - t$, $x(t)=t-t_n$.
The variable $x^\ast_t$ is called the amount of overjump (over the level $t$) and $x(t)$ is the amount of underjump (up to level $t$) -- see \cite[Ch. 2.7]{GK}.

In (\cite[\S 2.1]{Cox} or \cite[\S 2.1.1]{mitr}) the amount of overjump is called ``forward renewal time'' or ``residual life of the renewal interval'', and the amount of underjump is called ``backward renewal time'' or ``elapsed time of the renewal interval''.
\end{deff}
{\it In the sequel denote $ \overline{ \psi}(x) \stackrel{ \text{ \rm def}}{=\!\!\!=} 1- \psi(x)$ for all functions $ \psi(x)$.}
\subsection{ Smith's Key Renewal Theorem}
\begin{thm}[Key Renewal Theorem] \label{KeyTh}
Let $ \zeta$ be non-lattice random variable with distribution function $ \Phi(s) \stackrel{ \text{ \rm def}}{=\!\!\!=} \mathbf{P}\{ \zeta \leqslant s\}$, $ \mathbf{E}\, \zeta< \infty$, $H(t) \stackrel{ \text{ \rm def}}{=\!\!\!=} \sum \limits _{m=1}^ \infty \Phi^{m \ast}(t)$; let $b(s)$ be a directly Riemann integrable function on $[0, \infty)$.
Then \;
$$ \lim \limits _{t \to \infty} \int \limits _0^tb(t-s)\, \mathrm{d}H( s)= \frac{ \int \limits _0^ \infty b(s) \, \mathrm{d} s}{ \mathbf{E}\, \zeta^{ \phantom{1^1\!\!}}},
$$
see \cite[Theorem 35]{Serf}.
\end{thm}
\begin{rem}
If $F(s)$ is a distribution function, and $\int\limits _0^\infty s\,\mathrm{d}  F(s)<\infty$, then $(1-F(s))$ is directly Riemann integrable on $[0,+\infty)$.
\end{rem}

Smith's Key Renewal Theorem is very important for the renewal theory: for example, the following very useful fact (\cite[\S 5.2] {Cox}, \cite[\S 2.7.6]{GBS} ) is a consequence of the Smith's theorem; see also \cite{smith}.

\begin{prop} \label{Prop}
Consider the renewal process $0=t_0<t_1<\cdots<t_i<\cdots$, and let $x_t$ be the amount of underjump, and $x^ \ast_t$ be the amount of overjump.

If $F(s)$ is non-lattice and $ \int \limits _0^ \infty s \, \mathrm{d} F(s)< \infty$, then
\begin{equation} \label{Zv3}
\lim \limits _{t \to \infty} \mathbf{P}\{x_t>a\} = \lim \limits _{t \to \infty} \mathbf{P}\{x_t^ \ast>a\} = \frac{ \int \limits _a^ \infty \overline{F}(s) \, \mathrm{d} s}{ \int \limits _0^ \infty \overline{F}(s) \, \mathrm{d} s}.
\end{equation}
\end{prop}
\begin{cor}\label {1}
Hence,
$$
\lim \limits _{t \to \infty} \mathbf{E}\,x_t = \lim \limits _{t \to \infty} \mathbf{E} \,x_t^ \ast = {\displaystyle\int \limits _0^ \infty} \frac{ \int \limits _a^ \infty \overline{F}(s) \, \mathrm{d} s}{ \int \limits _0^ \infty \overline{F}(s) \, \mathrm{d} s} \, \mathrm{d} a= \frac{ \mathbf{E}\, \zeta^2}{2\, \mathbf{E}\, \zeta}.
$$
\end{cor}

Since processes describing the behaviour of queueing systems are often a combination of renewal processes {\it in some sense}, and asymptotic behaviour of these processes is very important, then renewal theory and, in particular, some facts similar to Proposition \ref{Prop} and Corollary \ref{1} are very significant.

For example, for studying single server queueing system behaviour, we can use {\it lineartype} or {\it linearwise} Markov processes (\cite{GK}) with state space $ \mathscr{X}=\{ \mathbb{Z}_+ \times \mathbb{R}_+\}$.
Linearwise process has two components.
The first component is equal to a quantity of customers in queueing system, and the second component is equal to the time elapsed from last change of the first component: $X_t=(n_t,x_t)$; let $ F_k (s) $ be a distribution function of time intervals when the process $(X_t,\,t\geqslant 0)$ is in the set $ \{n_t = k\}$, $k\in \mathbb{Z}_+$.

Therefore, the following conjecture seems natural: {\it if the distribution of a linearwise Markov process $(X_t,\,t\geqslant 0)$ (weakly) converges to stationary distribution as $ t \to \infty $, and if $\int\limits _0^\infty s \,\mathrm{d}  F_k (s)<\infty $ for all $k\in \mathbb{Z}_+$, then stationary distribution of the process $(X_t,\,t\geqslant 0)$ can be described by formulas similar to the formula (\ref{Zv3})}.

However, usually description of multichannel (parallel servers) queueing system behaviour in terms of linearwise process is impossible: in this case state space of the process describing queueing system behaviour has a more complex structure.
Nevertheless, B.A.~Sevastyanov (see \cite{Seva57}, \cite{Seva58}) proved an ergodic theorem for multichannel queueing system $M|G|n|0$ (Kendall's notation -- queueing system consists of $n$ servers, arrival flow is Poisson with parameter $\lambda$, independent service times have an arbitrary distribution function $G(s)$, and buffer size equals to zero).
B.A.~Sevastyanov defined the form of stationary distribution; this distribution is described by formulas similar to (\ref{Zv3}); see also \cite{Fortet}.

In some situations when service process has a stationary distribution, it is impossible to find any precise formula for this distribution: it is possible to find only an {\it estimate} for this stationary distribution.
However, corresponding estimate can be useful not only for solving practical problems related to optimization of queueing system work.
Estimation of stationary distribution also can be used to estimate convergence rate to stationary distribution for the distribution of process describing the behaviour of queueing system.

Below we give some facts, which can be used to obtain estimates of stationary distribution. Also these facts can be used for estimating the convergence rate of some queueing process distribution to this stationary distribution.

For example, in \cite{verezver} a strict (not only qualitative) estimate of convergence rate for availability factor has been obtained with the help of Theorem \ref{Th2} and Lemma \ref{lem1}.

These facts (Theorem \ref{Th2} and Lemma \ref{lem1}) seem quite natural, they are frequently mentioned in various publications, but the author did not see their complete proof. Therefore, the author propose her own proof of these ``folk'' propositions.

\section{Linearwise process}
\subsection{Definition of (one-dimensional) linearwise Markov process}

Linearwise (or lineartype) processes are very useful in many problems of queueing theory.
The fact is that the description of queueing system behaviour must contain a discrete component: for example, this discrete component can be a number of customers in the queueing system.
And time intervals between changes of these discrete components are random variables.
If these random variables have an exponential distribution, then the behaviour of queueing system is described by {\it continuous-time Markov chain.

But we refuse the condition of exponential distribution of time between state changes of the embedded Markov chain.
Thus, the time between changes of the discrete components of queueing process can be distributed {\it arbitrarily}.

Let us formalize the above.}

Let $ \mathfrak{P}=\|p_{i,j}\|$, $i,j \in \mathfrak{N} \subseteq \mathbb{Z}$ be the transition matrix, and let $ \Phi_k(s)= \mathbf{P}\{ \zeta_k \leqslant s\}$ be the distribution functions such that $ \Phi_k(0+)=0$, $k \in \mathfrak{N}$.
\begin{rem} If $ \Phi_k(0+)>0$, then we can modify the matrix $ \mathfrak{P}$ and the set $ \mathfrak{N}$ in such a way that replacement of the function $ \Phi_k(x)$ by the function $\displaystyle \frac{ \Phi_k(x)- \Phi_k(0+)}{1- \Phi_k(0+)}$ does not change the behaviour of the corresponding stochastic process (see below).
\end{rem}

Let for all $m \in \mathbb{N}$ and for all $k\in\mathfrak{N}$, $\zeta_k^{(m)} \stackrel{ \mathscr{D}}{=} \zeta_k$ be identically distributed random variables, and for $a \geqslant 0$ let $ \zeta_{k,\,a}$ be a random variable with the distribution function
$$
\mathbf{P} \{\zeta_{k,\,a}\leqslant s\}\stackrel{\text{\rm def}}{=\!\!\!=} \Phi_{k,\,a}(s)=1- \frac{ \overline{ \Phi}_k(s+a) \phantom{1^{1^1}\!\!\!\!}}{ \overline{ \Phi}_k(a)}= \mathbf{P}\{ \zeta_k \leqslant s+a| \zeta_k>a\};
$$
It is evident that $\zeta_{k,0}$
$\stackrel{{\mathscr D}}{=} \zeta_k$ (the designation ``$\stackrel{{\mathscr D}}{=}$'' means the same distribution).

If $ \Phi_k(a)=1$, then we assume $ \frac{ \overline{ \Phi}_k(s+a)}{ \overline{ \Phi}_k(a)}=0$.

Let all random variables $ \zeta_k^{(m)}$, $ \zeta_{k,\,a}$ be mutually independent ($k \in \mathfrak{N}$, $m \in \mathbb{Z}_+$, $a \in \mathbb{R}_+$); superscript of random variable $ \zeta_k ^ {(m)} $ indicates that this is $m$-th exemplar of random variable $ \zeta_k $; as already mentioned, $ \zeta_k^{(m)} \stackrel {\mathscr {D}} {=} \zeta_k $.

Denote $ \mathscr{X}=\{(n,x)\}= \mathfrak{N} \times \mathbb{R}_+$; let the set $ \mathscr{X}$ be equipped with standard Borel $ \sigma$-algebra $ \mathscr{B}(\mathscr{X})$.
The set $ \mathscr{X}=\{(n,x)\}$ is the state space of linearwise Markov process.
Let $S_i\stackrel{\text{\rm def}}{=\!\!\!=}\{(i, \cdot)\}\in \mathscr{B}(\mathscr{X})$ be called $i$-th level.

Consider  a Markov chain $ \mathcal{M}_k^{(r)}$ starting from the state $r$ (i.e. $ \mathcal{M}_0^{(r)}=r$).
Let transition matrix of $ \mathcal{M}_k^{(r)}$ be $ \mathfrak{P}$.
Again we suppose that $ \mathcal{M}_k^{(r)}$ and random variables $ \zeta_k^{(m)}$, $ \zeta_{k,a}$ are mutually independent.

Let us fix an arbitrary pair $(n_0,x_0) \in \mathscr{X}$ and denote
\begin{equation} \label{Zv4}
\begin{array}{c}
t_0 \stackrel{ \text{ \rm def}}{=\!\!\!=}0,\;\;\;
t_1 \stackrel{ \text{ \rm def}}{=\!\!\!=} \zeta_{n_0,x_0},\;\;\;
t_k \stackrel{ \text{ \rm def}}{=\!\!\!=} \zeta_{n_0,x_0}+ \sum \limits _{i=1}^k \zeta_{ \mathcal{M}_i^{(n_0)}}^{(i)}, \;\;\;
\nu_t \stackrel{ \text{ \rm def}}{=\!\!\!=} \max\{i:\;\;t_{i} \leqslant t\},
\\
\boxed{n_t \stackrel{ \text{ \rm def}}{=\!\!\!=} \mathcal{M}_{ \nu_t}^{(n_0)},\;\;\; x_t \stackrel{ \text{ \rm def}}{=\!\!\!=} t-t_{ \nu_t}.} \qquad \qquad
\end{array}
\end{equation}

\begin{deff}
{\it Linearwise Markov process with initial state} $X_0=(n_0,x_0)$ is a stochastic process $\big(X_t=(n_t,x_t),\,t\geqslant0\big)$ defined by formula (\ref{Zv4}); its values are in the state space $ \mathscr{X}$.
\end{deff}

{\it We skip the proof that the defined above process $\big(X_t,\,t\geqslant0\big)$ is a Markov.}

Denote $n(X_t) \stackrel{ \text{ \rm def}}{=\!\!\!=} n_t$, $x(X_t) \stackrel{ \text{ \rm def}}{=\!\!\!=} x_t$.
Here $n_t$ is called the {\it level $S_{n_t}$ where the process $\big(X_t,\,t\geqslant 0\big)$ is located at the time} $t$, $\nu_t$ is a moment of last change of process $(X_t,\,t\geqslant 0)$ level: $\nu_t\stackrel{\text{\rm def}}{=\!\!\!=}\sup\{s<t:\,n_s\neq n_t\}$.
And $x_t=(t-\sup\{s<t:\,n_s\neq n_t\})$ {\it is elapsed time of continuous being of process $\big(X_t,\,t\geqslant 0\big)$ in the level $n_t$ before the time $t$}; the times $t_k$ are the {\it times of pair $(n_t,x_t)=X_t$ change}, or {\it jumps } of the process $\big(X_t,\,t\geqslant 0\big)$; note that $x_{t_k+}=0$.

Construction (\ref{Zv4}) implies that at initial time $t=0$, $X_0=(n_0,x_0)$, i.e. the process $\big(X_t,\,t\geqslant 0\big)$ starts from the state $(n_0,x_0)$.

After the time $t=0$, through the time $ \zeta_{n_0,\,x_0}$, the process $\big(X_t,\,t\geqslant 0\big)$ gets into state $(i,0)\in S_i$ with probability $p_{n_0,\,i}$.
Then the process $\big(X_t,\,t\geqslant 0\big)$ stays in $S_i$ in the time $ \zeta_i^{(1)}$; then, at the time $( \zeta_{n_0,\,x_0}+ \zeta_i^{(1)}) $ the process $\big(X_t,\,t\geqslant 0\big)$ gets into state $(j,0)\in S_j$ with probability $p_{i,j}$, when it stays in the time $ \zeta_j^{(2)}$, etc.

We emphasize that random variable $ \zeta_ {n_0, \, x_0} $ is a residual time of the process $ X_t $ stay in the level $ S_ {n_0} $, under the condition that until time $ t = 0 $ the process $ X_t $ located {\it continuously} in the set $ S_{n_0} $ during the time $ x_0$.
We can interpret this situation as the start of process $\big(X_t,\,t\geqslant \Theta\big)$, $\Theta<0$, monitoring at the time $t=0$, and the process $\big(X_t,\,t\geqslant 0\big)$ is a sequel of process $\big(X_t,\,t\geqslant \Theta\big)$, $\Theta<0$, started in some previous time $\Theta$.

As already mentioned, linearwise processes $\big(X_t,\,t\geqslant 0\big)$ often occur in the study of queueing systems.
\begin{rem}
\emph {Usually} for embedded Markov chain $ \mathcal{M}_k^{(r)}$ of linearwise process $\big(X_t,\,t\geqslant 0\big)$ we have $p_{i,j}=0$ if $|i-j| \neq 1$.
This fact makes the study of the process $\big(X_t,\,t\geqslant 0\big)$ behaviour easier.
\end{rem}

\subsection{Ergodic Theorem for linearwise Markov process}
\subsubsection{Conditions}
\begin{enumerate}
\item[ \textbf{1.}] {\it Condition for Markov chain $ \mathcal{M}_n$:}

$ \forall i,j \in \mathfrak{N}$ \;
$ \mathbf{P} \left\{ \lim \limits _{k \to \infty} \frac{ \sum \limits _{m=1}^k \mathbf{1} \left\{ \mathcal{M}_m^{(j)}=i \right\}}{k}=p_i \right\}=1$; $ \sum \limits _{i \in \mathfrak{N}}p_i=1$;
\\
let $ \mathfrak{N}_s \stackrel{ \text{ \rm def}}{=\!\!\!=}\{i \in \mathfrak{N}:\;p_i>0\}$ be a set of essential states of the Markov chain $ \mathcal{M}_n$.

\item[ \textbf{2.}] Conditions for $ \Phi_i(s)$:
\begin{enumerate}
\item[ \qquad \emph{a.}] $ \mathbf{E}\, \zeta_{i}=T_i< \infty$ for all $i \in \mathfrak{N}$;
\item[ \qquad \emph{b.}] $ \sum \limits _{i \in \mathfrak{N}}p_iT_i= \sum \limits _{i \in \mathfrak{N}_s}p_iT_i=T< \infty$; \;

\item[ \qquad \emph{c.}] $ \mathbf{supp\,}(X_t) \stackrel{ \text{ \rm def}}{=\!\!\!=} \bigcup \limits_{{i \in \mathfrak{N}_s}} \mathbf{supp\,} (\mathscr{L}( \zeta_i)) \not \subseteq \{a+b \mathbb{Z}\} \stackrel{ \text{ \rm def}}{=\!\!\!=}\{a+bn,n \in \mathbb{Z}\}$ for all $ a,b \in \mathbb{R}$, where $ \mathbf{supp\,}(\mathscr{L}(\zeta_i))$ is a support of random variable $ \zeta_i$ distribution.
\end{enumerate}
\end{enumerate}
\begin{rem}~
\begin{enumerate}
\item Condition \textbf{2.}\emph{c.} implies $ \sum \limits _{i \in \mathfrak{N}} \mathbf{D}\, \zeta_i>0 $.
This ensures that the process $\big(X_t,\,t\geqslant 0\big)$ is \emph {stochastic}.

{\it We do not assume that} $\mathsf E  (\zeta_i)^2<\infty$.
\item Usually, for arbitrary Markov chain $ \mathcal{M}_n$ the verification of condition \textbf{1.\,} is difficult or even impossible.

However, as a rule, the condition \textbf{1.\,} can be verified for queueing system.

\emph{Usually}, we can provide upper bounds for the values $p_i$ for queueing system linearwise process; as a rule, this estimation is enough to estimate the stationary distribution.
\item {\it Usually}, for many queueing systems checking the condition \textbf{2.}\emph{b.} is not difficult.
\end{enumerate}
\end{rem}
\subsubsection{Theorem \ref{Th2}}
\begin{thm} \label{Th2}
~

\noindent \textbf{1.} Conditions 1--2 are necessary and sufficient for the existence of stationary distribution of the process $\big(X_t,\,t\geqslant 0\big)$.

\noindent \textbf{2.} Conditions 1--2 are necessary and sufficient for weak convergence of process $\big(X_t,\,t\geqslant 0\big)$ distribution to this stationary distribution.

\noindent \textbf{3.} Also, for stationary distribution $\mathscr{P}$ of the process $\big(X_t,\,t\geqslant 0\big)$ we have:
\begin{equation} \label{Zv6}
\mathscr{P}\{\widetilde n=i,\;\widetilde x> a\;\&\;\widetilde x^ \ast>b\}= \lim \limits _{t \to \infty} \mathbf{P}\{n_t=i,\;x_t> a\;\&\;x^ \ast_t>b\}= \frac{p_i}{T} \int \limits _{a+b}^ \infty \overline{ \Phi}_i(u) \, \mathrm{d} u,
\end{equation}
where $x_t^ \ast \stackrel{ \text{ \rm def}}{=\!\!\!=} \left(\inf\{s>t:\;n_s \neq n_t\}-t \right)$ is {\it amount of overjump}, or residual time of process $\big(X_t,\,t\geqslant 0\big)$ continuous being in the level $S_{n_t}$ after the time $t$, $i \in \mathfrak{N}$.
\end{thm}
\begin{rem}
The distribution given in (\ref{Zv6}) is stationary for the case, where $ \mathbf{supp\,}(X_t) \subseteq\{a+b \mathbb{Z}\}$ (naturally, with slight modifications), but now we do not have sufficiently short and easy to read proof of this fact.
\end{rem}
\subsubsection{Proof of Theorem \ref{Th2}}
\begin{proof} [Proof]~

\noindent \textbf{1.} The necessity of conditions \textbf {1.} and \textbf {2.} is evident.

\noindent \textbf{2.} Condition \textbf{1.} implies existing of single indecomposable class of Markov chain $ \mathcal{M}_n$. This indecomposable class is a set $ \mathfrak{N}_s$ of essential states of Markov chain $ \mathcal{M}_n$.
This indecomposable class $ \mathfrak{N}_s \subseteq \mathfrak{N}$ is ergodic or periodic.

Naturally, in the sequel we will consider essential states of Markov chain $ \mathcal{M}_n$ only.

\noindent \textbf{3.} Conditions \textbf{1.}, \textbf{2.}\emph{a.} and \textbf{2.}\emph{b.} imply
\begin{equation} \label{Zv7}
\lim \limits _{t \to \infty} \mathbf{P}\{n_t=k\}= \frac{p_kT_k}{T}.
\end{equation}

\noindent \textbf{4.} Conditions \textbf{1.} and \textbf{2.}\emph{c.} imply:
there exist $q_1,q_2 \in \mathfrak{N}_s$ such that random variable $ \zeta_{q_1}+ \zeta_{q_2}$ is non-lattice.

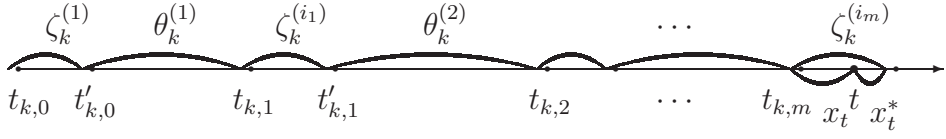
\begin{figure}[h]
\begin{center}
\begin{picture}(360,40)
\qbezier(0,20)(14,30)(28,20)
\qbezier(0,20)(14,31)(28,20)
\qbezier(0,20)(14,32)(28,20)
\qbezier(28,20)(58,30)(88,20)
\qbezier(28,20)(58,31)(88,20)
\qbezier(28,20)(58,32)(88,20)
\qbezier(88,20)(104,30)(120,20)
\qbezier(88,20)(104,31)(120,20)
\qbezier(88,20)(104,32)(120,20)
\qbezier(120,20)(160,30)(200,20)
\qbezier(120,20)(160,31)(200,20)
\qbezier(120,20)(160,32)(200,20)
\qbezier(200,20)(213,30)(226,20)
\qbezier(200,20)(213,31)(226,20)
\qbezier(200,20)(213,32)(226,20)
\qbezier(226,20)(261,30)(296,20)
\qbezier(226,20)(261,31)(296,20)
\qbezier(226,20)(261,32)(296,20)
\qbezier(296,20)(308,10)(320,20)
\qbezier(296,20)(308,9)(320,20)
\qbezier(296,20)(308,8)(320,20)
\qbezier(320,20)(326,10)(332,20)
\qbezier(320,20)(326,9)(332,20)
\qbezier(320,20)(326,8)(332,20)
\qbezier(296,20)(314,30)(332,20)
\qbezier(296,20)(314,31)(332,20)
\qbezier(296,20)(314,32)(332,20)
\put(0,20){ \vector(1,0){350}}
\put(0,20){ \circle*{2}}
\put(0,5){$t_{k,0}$}
\put(14,33){$ \zeta_k^{(1)}$}
\put(28,20){ \circle*{2}}
\put(25,5){$t_{k,0}'$}
\put(55,33){$ \theta_k^{(1)}$}
\put(88,20){ \circle*{2}}
\put(85,5){$t_{k,1}$}
\put(101,33){$ \zeta_k^{(i_1)}$}
\put(120,20){ \circle*{2}}
\put(117,5){$t_{k,1}'$}
\put(157,33){$ \theta_k^{(2)}$}
\put(200,20){ \circle*{2}}
\put(197,5){$t_{k,2}$}
\put(245,33){$ \cdots$}
\put(226,20){ \circle*{2}}
\put(245,5){$ \cdots$}
\put(296,20){ \circle*{2}}
\put(308,0){$x_t$}
\put(285,5){$t_{k,m}$}
\put(326,0){$x_t^ \ast$}
\put(311,33){$ \zeta_k^{(i_{m})}$}
\put(332,20){ \circle*{2}}
\put(318,5){$t$}
\put(316,20){ \circle*{3}}
\end{picture}
\caption{Times of stay of the process $\big(X_t,\,t\geqslant 0\big)$ in the level $S_k$ and out of the level $S_k$.}
\end{center}
\end{figure}

\noindent \textbf{5.} Let us fix an arbitrary $k \in \mathfrak{N}_s$ and (at the beginning of proof) let $X_0=(k,0)$.
Denote
$$
\begin{array}{l}
t_{k,0} \stackrel{ \text{ \rm def}}{=\!\!\!=}0, \quad t_{k,\,0}' \stackrel{ \text{ \rm def}}{=\!\!\!=} \inf\{t \geqslant t_{k,\,0}:\,n_t \neq k\};
\\
t_{k,j} \stackrel{ \text{ \rm def}}{=\!\!\!=} \inf\{t \geqslant t_{k,j-1}':\, n_t=k\}, \quad t_{k,j}' \stackrel{ \text{ \rm def}}{=\!\!\!=} \inf\{t \geqslant t_{k,j}:\,n_t \neq k\};
\\
\theta_k^{(j)} \stackrel{ \text{ \rm def}}{=\!\!\!=} t_{k,j}-t_{k,j-1}', \quad \zeta_k^{i_j}=t'_{k,j}-t_{k,j}.
\end{array}
$$
Here $t_{k,j}$ is the time when process $\big(X_t,\,t\geqslant 0\big)$ hits into the level $S_k$; and $t'_{k,j}$ is the time when the process $\big(X_t,\,t\geqslant 0\big)$ leaves the level $S_k$; and $ \theta_k^{(i)}$ is the length of time interval where $X_t\notin S_k$: in other words, $ \theta_k^{(i)}$ is the time between two subsequent stays of the process $\big(X_t,\,t\geqslant 0\big)$ in the set $S_k$.

Evidently, random variables $ \theta_k^{(i)}$ are i.i.d.
Denote by $ \Psi_k(s) \stackrel{ \text{ \rm def}}{=\!\!\!=} \mathbf{P} \left\{ \theta_k^{(i)} \leqslant s \right\}$ the distribution function of $ \theta_k^{(i)}$.

Then $ \hat{ \theta}_k^{\,(j)}= \theta_k^{(j)}+ \zeta_k^{(i_j)}$ is the time between two consequent (adjacent) hits of process $\big(X_t,\,t\geqslant 0\big)$ into the state $(k,0)$.
The random variable $ \hat{ \theta}_k^{\,(j)}$ has the distribution function $ \hat{ \Psi}_k(s)= \Psi_k \ast \Phi_k(s)$. Note that the random variable $ \hat{ \theta}_k^{\,(j)}$ is non-lattice, since
$$
\mathbf{P} \left\{{ \mbox{\it The paths of the Markov chain $ \mathcal{M}_n$ from state $\{k\}$ }\atop\mbox{\it to the same state $\{k\}$ contains the states $\{q_1\}$ and $\{q_2\}$}}\right\}>0,
$$
-- see condition \textbf{4.}

We emphasize that the process $\big(X_t,\,t\geqslant 0\big)$ is a regenerative process, and times $t_{k,\,i}$, $k\in\mathbb{N}$, are regeneration points of the process $\big(X_t,\,t\geqslant 0\big)$.

\noindent \textbf{6.} Let us fix arbitrary time $ t$ and let us find $ \mathbf{P}\{n_ t=k,x_ t> a,x_t^ \ast>b\}$.
Denote
\begin{equation*}
\begin{array}{l}
E_k(t,m,u, \, \mathrm{d} u)
\stackrel{ \text{ \rm def}}{=\!\!\!=} \left\{{ {\mbox{{\it Last hit of the process $\big(X_t,\,t\geqslant 0\big)$ to the level $S_k$ before} }} \atop{\mbox{{\it the time $t$ is the time $t_{k,\,m}$, and $t_{k,\,m}\in (u,\, u+ \, \mathrm{d} u)$}} }} \right\};
\\
\\
E_k(t,m) \stackrel{ \text{ \rm def}}{=\!\!\!=}\{ \mbox{{\it Before the time $t$ there was exactly $m$ hits of the process $\big(X_t,\,t\geqslant 0\big)$ into the level $S_k$}}\};
\\
\\
E_k(t,0)=\{\mbox{$X_t\in S_k$ {\it for all }$t\in(0,t)$}\}.
\end{array}
\end{equation*}

Easy to see that
$$
\left\{E_k(t,m,u, \, \mathrm{d} u)\;\&\;x_t> a\;\&\; x_t^ \ast>b \right\} \Leftrightarrow \left\{u+ \, \mathrm{d} u \leqslant t-a \;\&\; \zeta_k^{(i_m)}> t-u+b \right\},
$$
and for $u \leqslant t-a$ we obtain
\begin{equation*}
\begin{array}{l}
\mathbf{P}\{n_ t=k,x_ t> a\;\&\;x_t^ \ast>b\;\&\;E_k(t,0)\} = \overline{ \Phi}_k(t);
\\
\mathbf{P}\{n_ t=k,x_ t> a, x_t^ \ast>b\; \&\;E_k(t,m,u, \, \mathrm{d} u)\}
= \overline{ \Phi}_k(t-u+b) \, \mathrm{d} \hat{ \Psi}_k^{m \ast}(u);\\
\mathbf{P}\{n_t=k,x_ t> a, x_t^ \ast>b\;\&\;E_k(t,m)\}= \int \limits _0^{t-a} \overline{ \Phi}_k(t-u+b) \, \mathrm{d} \hat{ \Psi}_k^{m \ast}(u).
\end{array}
\end{equation*}
Now denote $H_k(u) \stackrel{ \text{ \rm def}}{=\!\!\!=} \sum \limits _{m=1}^ \infty \hat{ \Psi}_k^{m \ast}(u)$ and $r\stackrel{\text{\rm def}}{=\!\!\!=} t-a$. Thus, we get:
\begin{multline} \label{Zv8}
\mathbf{P}\{n_ t=k,x_ t> a, x_t^ \ast>b\}
= \mathbf{P}\{n_ t=k,x_ t> a, x_t^ \ast>b\;\&\;E_k(t,0)\}+\\
+ \sum \limits _{m=1}^ \infty \mathbf{P}\{n_ t=k, x_ t> a, x_t^ \ast>b\;\&\;E_k(t,m)\}
= \overline{ \Phi}_k(t)+ \sum \limits _{m=1}^ \infty \int \limits _0^{t-a} \overline{ \Phi}_k(t-u+b) \, \mathrm{d} \hat{ \Psi}_k^{m \ast}(u)=\\
= \overline{ \Phi}_k(t)+ \int \limits _0^r \overline{ \Phi}_k(r+a+b-u) \, \mathrm{d} H_k(u).
\end{multline}

Finally, let us find the limit in the formula (\ref{Zv8}) as $t \to \infty$. Using Smith's Theorem \ref{KeyTh} for function $b(r) \stackrel{ \text{ \rm def}}{=\!\!\!=} \overline{ \Phi}_k(r+a+b)$, we get:
\begin{equation}\label{fin}
\lim \limits _{t \to \infty} \mathbf{P}\{n_ t=k,x_ t> a, x_t^ \ast>b\}= \frac{ \int \limits _0^ \infty \overline{ \Phi}_k(r+a+b) \, \mathrm{d} r} { \int \limits _0^ \infty r \, \mathrm{d} H_k(r)}= \frac{ \int \limits _{a+b}^ \infty \overline{ \Phi}_k(r) \, \mathrm{d} r}{ \phantom{ \frac 11} \mathbf{E}\, \hat{ \theta}_k \phantom{^{1^{1}}\!\!\!}}.
\end{equation}
Then, combining (\ref{Zv7}) and (\ref{fin}), we obtain:
$$\lim \limits _{t \to + \infty} \mathbf{P}\{n_t=k\}= \frac{p_kT_k}{T}= \lim \limits _{t \to + \infty} \mathbf{P}\{n_t=k, x_t> 0,x_t^ \ast>0\}
= \frac{ \int \limits _0^ \infty \overline{ \Phi}_k(r) \, \mathrm{d} r}{ \phantom{ ^{1^1}\!\!\!\!} \mathbf{E}\, \hat{ \theta}_k \phantom{ \frac 11} }= \frac{T_k}{ \phantom{ \frac 11} \mathbf{E}\, \hat{ \theta}_k \phantom{ ^{1^1}\!\!\!\!} } \Rightarrow{ \mathbf{E}\, \hat{ \theta}_k}= \frac{T}{p_k},
$$
and
$$
\lim \limits _{t \to+ \infty} \mathbf{P}\{n_ t=k,x_ t> a, x_t^ \ast>b\}= \frac{p_k }{T} \int \limits _{a+b}^ \infty \overline{ \Phi}_k(r) \, \mathrm{d} r.
$$

\noindent \textbf{7.} Now let $X_0$ be {\it arbitrary}.
Denote $ \Theta_k(X_0) \stackrel{ \text{ \rm def}}{=\!\!\!=} \inf\{t>0:\,X_t=(k,0)\}$; it is easy to see that $\Theta_k(X_0)$ is a time of the first hit of the process $\big(X_t,\,t\geqslant 0\big)$ into state $(k,0)\in S_k$; this is the {\it first} regeneration time.

Conditions \textbf{1.} and \textbf{2.}\emph{a.} imply $\mathbf{P} \{\Theta_k(X_0)<\infty\}=1$, since $ \mathbf{E}\, \Theta_k(X_0)< \infty$.
Therefore for all initial states $X_0$ we have:
\begin{multline*}
\lim \limits _{t \to + \infty} \mathbf{P}_{X_0}\{n_t=k, x_ t> a, x_t^ \ast>b\}=
\\
= \lim \limits _{t \to + \infty} \sum \limits _{n=1}^ \infty \mathbf{P}_{X_0}\{n_t=k, x_ t> a, x_t^ \ast>b\;\&\; \Theta_k(X_0) \in(n-1,n]\}=\hspace{5cm}
\\
= \sum \limits _{n=1}^ \infty \left(\lim \limits _{t \to + \infty} \mathbf{P}_{X_0}\{n_t=k,x_ t> a, x_t^ \ast>b|\, \Theta_k(X_0) \in(n-1,n]\} \right) \mathbf{P}_{X_0} \{ \Theta_k(X_0) \in(n-1,n]\}=
\\
= \frac{p_k \int \limits _{a+b}^ \infty \overline{ \Phi}_k(r) \, \mathrm{d} r}{T} \sum \limits _{n=1}^ \infty \mathbf{P}_{X_0}\{ \Theta_k(X_0) \in(n-1,n]\} = \frac{p_k \int \limits _{a+b}^ \infty \overline{ \Phi}_k(r) \, \mathrm{d} r}{T}.
\end{multline*}
The Theorem \ref{Th2} is proved.
\end{proof}
\begin{rem}
Theorem \ref{Th2} is a generalization of Theorem \cite[\S 2.6]{Klim}, where there are three variants of \emph{sufficient} conditions for the proof of formula (\ref{Zv6}).
\end{rem}
\begin{rem}
Using computations on the proof of Theorem \ref{Th2} we can prove a fact very useful for applications of the renewal theory. This fact, in some sense, is ``folklore'': many queueing theory experts know this fact, but we could not find when and who proved it.
\end{rem}
\section{Expectation of amount of overjump}
Now let $\big(X_t,\,t\geqslant 0\big)$ be renewal process with non-lattice renewal period $ \zeta^{(k)} \stackrel{ \mathscr{D}}{=} \zeta$; let the distribution function of $ \zeta^{(k)}$ be $ \Phi(s)$; let $t_k= \sum \limits _{i=0}^k \zeta^{(i)}$ be a renewal points.

{\it Now we suppose} $ \mathbf{E}\, \zeta^2< \infty$.

Renewal process is a special case of linearwise Markov process, where embedded Markov chain has the single state only; stochastic matrix of this Markov chain is $\mathfrak{P}=(1)$.
We replace the state space of the process $ \mathscr {X} =\{1\}\times \mathbb {R} _ + $ by $ \mathscr {X} =\mathbb {R} _ + $ (we reduce the full states space of the process $\big(X_t,\,t\geqslant 0\big)$ by dropping not changing first component).
The amount of overjump $x^ \ast_t$ of renewal process at the time $t$ is the time elapsing from the time $t$ to the next renewal ($t_i$): $x^ \ast_t \stackrel{ \text{ \rm def}}{=\!\!\!=}( \min\{t_i:\,t_i>t\}-t)$.

It is well known (\cite{GBS}, \cite{GK}), that if the distribution of $\zeta$ is non-lattice and $\mathsf E \,\zeta<\infty$, then $ \lim \limits _{t \to \infty} \mathbf{P}\{x_t^ \ast>s\}= \mathbf{P}\{ \widetilde{x}_t^{\, \ast}>s\}= {\displaystyle}\frac{ \int \limits _s^ \infty \overline{ \Phi}(u) \, \mathrm{d} u}{ \mathbf{E}\, \zeta}$; random variable $ \widetilde{x}_t^ \ast$ is called stationary amount of overjump, and $ \mathbf{E}\, \widetilde{x}_t^ \ast= \displaystyle\frac{ \mathbf{E}\, \zeta^2}{2 \mathbf{E}\, \zeta}$.
\begin{lem}\label{lem1}
If the distribution of $\zeta$ is non-lattice, and $\mathsf E \,\zeta^2<\infty$, then $ \mathbf{E}\,x^ \ast_t \nearrow \mathbf{E}\, \widetilde{x}_t^ \ast= {\displaystyle}\frac{ \mathbf{E}\, \zeta^2}{2 \mathbf{E}\, \zeta}$.
\end{lem}
\begin{proof}[Proof]
Usually (\cite[ \S 2.3]{GBS}, \cite[ \S 2.6]{GK}) the proof of formula (\ref{Zv3}) follows from the equality
\begin{equation}\label{Zv55}
R(s,t) \stackrel{ \text{ \rm def}}{=\!\!\!=} \mathbf{P}\{x_ t^ \ast> s\} = \overline{ \Phi}(t+s)+ \int \limits _0^{t} \overline{ \Phi}(t-u+s) \, \mathrm{d} H(u).
\end{equation}
This equality (\ref{Zv55}) is a special case of (\ref{Zv8}), where $H(u)\stackrel{\text{\rm def}}{=\!\!\!=} \sum \limits _{i=1}^ \infty \Phi^{ \ast i}(u)$ is (nondecreasing) renewal function.
For $ \Delta>0$ we have:
\begin{multline*}
R(s,t+ \Delta)-R(s,t)= \Phi(t+s)- \Phi(t+ \Delta+s)+
\\
+ \int \limits _0^{t} \Phi(t+ \Delta-u+s)- \Phi(t-u+s) \, \mathrm{d} H(u)+ \int \limits _t^{t+ \Delta} \overline{ \Phi}(t+ \Delta-u+s) \, \mathrm{d} H(u) \geqslant0,
\end{multline*}
i.e. 
\begin{equation}\label{ner}
  R(s,t+ \Delta)=\mathbf{P} \{x^\ast_{t+\Delta}>s\} \geqslant R(s,t)=\mathbf{P} \{x^\ast_{t}>s\},
\end{equation} 
therefore for $t_1<t_2$
$$ \mathbf{E}\,x_ {t_1}^ \ast= \int \limits _0^ \infty R(s,t_1) \, \mathrm{d} s \leqslant \int \limits _0^ \infty R(s,t_2) \, \mathrm{d} s= \mathbf{E}\,x_ {t_2}^ \ast \leqslant \lim \limits _{t \to \infty} \mathbf{E}\,x_ {t}^ \ast={ \int \limits _0^ \infty} \left( \frac{ \int \limits _s^ \infty 1- \Phi(u) \, \mathrm{d} u}{ \int \limits _0^ \infty 1- \Phi(u) \, \mathrm{d} u} \right) \, \mathrm{d} s= \frac{ \mathbf{E}\, \zeta^2}{2 \mathbf{E}\, \zeta}.
$$
The Lemma \ref{lem1} is proved.
\end{proof}
\begin{rem} In addition, from (\ref{ner}) we have:
\begin{enumerate}
  \item If $\mathsf E \,\zeta^k<\infty$ for $k>2$, then
      \begin{multline*}
        \mathsf E (x_t^\ast)^{k-1}=(k-1)\int\limits _0^\infty s^{k-2} R(s,t) \,\mathrm{d}  s \leqslant (k-1)\int\limits _0^\infty s^{k-2}\left(\frac{\int\limits _s^ \infty 1- \Phi(u) \, \mathrm{d} u}{\mathsf E \,\zeta}\right)\,\mathrm{d}  s=\\
        =\frac{\int\limits _0^\infty s^{k-1}\,\mathrm{d}  \int\limits _s^\infty (1-\Phi(u))\,\mathrm{d}  u}{\mathsf E \,\zeta}=\frac{\mathsf E \,\zeta^k}{k\mathsf E \,\zeta}.
      \end{multline*}
  \item If $\mathsf E \,e^{\alpha\zeta}<\infty$ for $\alpha>0$, then
  \begin{multline*}
    \mathsf E \,e^{\alpha\, x_t^\ast}=\alpha\int\limits _0^\infty e^{\alpha s}R(s,t)\,\mathrm{d}  s\leqslant\frac{\alpha}{\mathsf E \,\zeta}\int\limits _0^\infty e^{\alpha s}\int\limits _s^\infty (1-\Phi(u))\,\mathrm{d}  u \,\mathrm{d}  s =\\
    \frac{1}{\mathsf E \,\zeta}\left(-\mathsf E \,\zeta+\int\limits _0^\infty e^{\alpha s}(1-\Phi(s))\,\mathrm{d}  s\right)=\frac{\mathsf E \,e^{\alpha\zeta}}{\alpha\mathsf E \,\zeta}-1.
  \end{multline*}
\end{enumerate}
\end{rem}
\begin{rem}
Recall that $x_t$ is the time elapsed from the previous renewal to $t$, i.e. $x_t\stackrel{\text{\rm def}}{=\!\!\!=}(t-\max\{t_i:\,t_i<t\})$.

Random variable $ \widetilde{x}_t$ is called stationary amount of underjump, if $ \lim \limits _{t \to \infty} \mathbf{P}\{x_t>s\}= \mathbf{P}\{ \widetilde{x}_t>s\}$. From Theorem 2 we have $ \mathbf{P}\{ \widetilde{x}_t>s\}={\displaystyle}\frac{ \int \limits _s^ \infty \overline{ \Phi}(u) \, \mathrm{d} u}{ \mathbf{E}\, \zeta}$.

The amount of underjump $x_t$ is a single component of the process $\big(X_t,\,t\geqslant 0\big)$.
And for $x_t$ a similar proposition is true: if the distribution of $\zeta$ is non-lattice, and $\mathsf E \,\zeta^2<\infty$, then $ \mathbf{E}\,x_t \nearrow \mathbf{E}\, \widetilde{x}_t= {\displaystyle}\frac{ \mathbf{E}\, \zeta^2}{2 \mathbf{E}\, \zeta}$.

The proof is identical to the proof of Lemma \ref{lem1}.

\end{rem}

\begin{cor}
If we assume that \; $ \mathbf{E}\,( \zeta_i)^2< \infty$ in the conditions of Theorem \ref{Th2}, then for all $ \tau> \inf\{t \geqslant 0:\;X_t=(k,0)\}$ the inequalities
\begin{equation}\label{OcEst}
\mathbf{E}\,(x_ \tau|n_ \tau=k) \leqslant {\displaystyle}\frac{ \mathbf{E}( \zeta_k)^2} {2 \mathbf{E}\, \zeta_k};\qquad \mathbf{E}\,(x_ \tau^ \ast|n_ \tau=k) \leqslant {\displaystyle}\frac{ \mathbf{E}(\zeta_k)^2} {2 \mathbf{E}\, \zeta_k}
\end{equation}
hold.

These inequalities give the possibility to obtain not only estimation of linearwise process stationary distribution, but also an estimation of linearwise process distribution at the arbitrary (big enough) time. Naturally, we can not use (\ref{OcEst}) before the first regeneration time.
\end{cor}

\section{Multidimensional piecewise-linear Markov processes}
Now the great interest of queueing theory is the study of multichannel queueing system and of queuing network.
Let some queueing system (or queuing network) consist of $ N \leqslant \infty $ servers and incoming flow of a customers can be also multichannel, i.e. incoming flow includes $ K >0 $ (obviously independent) incoming flow of customers; obviously all incoming flows are independent.

Assume that the time interval between customer arrivals of $ i $-th incoming flow has distribution function $ \Phi_i (s) $, in the $ i $-th server the time of the service has distribution function $ F_i (s) $.
We can also consider some additional conditions, such as:

-- customers can have different types of service;

-- some customers may be ``impatient'' (i.e., this customer can leave the queueing system before completion of service if (possibly random) residence time in the system has ended);

-- and so on.

The behaviour of multichannel and multiserver system can be described by Markov process $\big(X_t,\,t\geqslant 0\big)$ with the state space $ \mathscr{X} \subseteq \mathbb{R}_+^{K+1} \bigcup \left(\bigcup \limits_{n=1}^ \infty \bigcup \limits_{m=K+1}^ \infty\{ \mathbb{N}^n \times \mathbb{R}_+^{m}\} \right)= \mathscr{X}_0 \bigcup \left(\bigcup \limits_{n=1}^ \infty \mathscr{X}_{n,m} \right) $.

Here we show a set which is knowingly sufficient to describe the behaviour of the complex queueing system: in fact the state space $ \mathscr {X} $ may be easier.

For example, the state of process $\big(X_t,\,t\geqslant 0\big)$ at the time $t$ can be described by vector \linebreak $ X_t= \left(\mathcal{N}_t, \mathcal{M}_t,y_t^{(1)}, \ldots,y_t^{(K)}, x_t^{(1)}, \ldots,x_t^{( \mathcal{M}_t)},z_t^{(1)}, \ldots,z_t^{(R_t)}, n_t^{(1)}, \ldots,n_t^{(Q_t)} \right)$.
Here $ \mathcal{N}_t$ is a number of customers in queueing system, $ \mathcal{M}_t$ is a number of customers served at time $t$; $y_t^{(j)}$ is the time from the last arrival of the customer from $j$-th incoming flow; $x_t^{(j)}$ is elapsed time of service on $j$-th server, i.e. elapsed time of service of customer located on $j$-th server; also we can include into vector $ X_t$ service waiting time of customers or duration of stay in the system for all customers; total length of the queue of queueing system or queue length on each server of queueing system; possibility of packaged service of incoming customers, and so on.: some of this parameters are presented in parts of the vector $X_t$ -- that is $ \left(z_t^{(1)}, \ldots,z_t^{(R_t)} \right)$ (elapsed server time) and $ \left(n_t^{(1)}, \ldots,n_t^{(Q_t)} \right)$ (queue length).

Continuous components of process $\big(X_t,\,t\geqslant 0\big)$ change linearly between the times of process $\big(X_t,\,t\geqslant 0\big)$ discrete components change: at the time of discrete components change some of continuous components leave the vector $X_t$, and some of new continuous components may appear in the vector $X_t$.
For each specific queuing systems we can define the transition probabilities for the vector $X_t$: process $\big(X_t,\,t\geqslant 0\big)$ is Markov for a suitable choice of the vector $X_t$ component; we shall say that this process $\big(X_t,\,t\geqslant 0\big)$ is called \emph {piecewise-linear Markov process}.

The ``level'' of such \emph {piecewise-linear Markov process} (as described above) can be defined as $\mathcal N_t$, or the pair $(\mathcal{N}_t, \mathcal{M}_t)$, or more comprehensive set of discrete components of the vector $ X_t$.

Then, in general case, the embedded chain of ``levels'', in which process can be staying, is {\it not homogeneous Markov chain}.
However, in some cases it is possible to define \emph{upper bound} of this embedded non-homogeneous chain transition probabilities by transition probabilities of some homogeneous Markov chain.

Indeed, to find such a bound, for example, we can consider another {\it auxiliary} queueing system having embedded homogeneous Markov chain; but the parameters of this new queueing system are the bounds for the parameters of original queueing system.

Further, after studying this new auxiliary queueing system, we can prove the ergodic theorem for auxiliary queueing system. Then, we can obtain formulas of auxiliary queueing system stationary distribution.
Using these formulas, we can find bounds for stationary distribution of original queueing system; in fact, we prove the ergodic theorem for original queueing system.

Moreover, using these bounds we can find an estimate for convergence rate of original queueing system distribution to stationary distribution.

Also finding of such bounds for stationary distribution of multidimensional piecewise-linear Markov processes is usually possible for the following reason.
Usually, incoming flow of queueing system is such that only transitions $ \mathcal{N}_t \mapsto \mathcal{N}_t+1$, $ \mathcal{M}_t \mapsto \mathcal{M}_t+1$, $ \mathcal{N}_t \mapsto \mathcal{N}_t- \ell$, $ \mathcal{M}_t \mapsto \mathcal{M}_t- \ell$ \; ($ \ell \geqslant 1$) are possible (only these transitions have positive probability).

So, let for some $\widetilde{ \mathcal{N}}$ the inequality $p_n= \mathbf{P}\{n \mapsto n+1\} \leqslant \mathfrak{p}< \frac12$ is satisfied for all $n\geqslant \widetilde{ \mathcal{N}}$.
And let the set of distributions $ \Phi_1, \ldots, \Phi_K,F_1, \ldots, F_{ \widetilde{ \mathcal{N}}}$ has common non-lattice support.
Then we can establish: for any initial state (or initial distribution) of the process $\big(X_t,\,t\geqslant 0\big)$, the distribution of this process $\big(X_t,\,t\geqslant 0\big)$ converges weakly to stationary distribution $ \mathscr{P}$.

Now, we can use the technique machinery applied for the proof of Theorem \ref{Th2}, and also we can apply well-known facts from the random walk theory. This way, we can obtain estimates for stationary distribution, for example, such that:
$$
\mathscr{P}(A(k)) \leqslant (2 \pi)^{k- \widetilde{ \mathcal{N}}} \mbox { for } k> \widetilde{ \mathcal{N}};
\;\;
\mathscr{P}(B(k,m,a_1, \ldots,a_m)) \leqslant \frac{ \prod \limits_{i=1}^m \int \limits _{a_i}^ \infty(1-F_i(s)) \, \mathrm{d} s}{ \int \limits _0^ \infty \prod \limits_{i=1}^m(1-F_i(s)) \, \mathrm{d} s} \mbox { for }k \leqslant{ \mathcal{N}},
$$
where $A(k)\stackrel{\text{\rm def}}{=\!\!\!=}\{X \in \mathscr{X}:\, \mathcal{N}=k\}$, and $B(k,m,a_1, \ldots,a_m)\stackrel{\text{\rm def}}{=\!\!\!=}\{X \in \mathscr{X}:\, \mathcal{N}=k, \mathcal{M}=m,x_1>a_1, \ldots,x_m>a_m\}$,

Such inequalities can be applied for estimating the convergence rate of process $\big(X_t,\,t\geqslant 0\big)$ distribution to stationary distribution $ \mathscr{P}$ of the process $\big(X_t,\,t\geqslant 0\big)$, and for estimating some parameters of queueing system behaviour in the long range of time.

\textbf{Acknowledgements. } The author is grateful to A.Yu.Veretennikov for constant attention to this work, and to V.V.Kozlov for useful discussion, and to Yu.S.~Semenov for invaluable help.

This research was supported by RFBR (project No 14-01-00319 A).

\end{document}